\documentclass[oneside, 12pt]{amsart}
\usepackage{amscd, amssymb, amsmath, mathrsfs}
\usepackage[english]{babel}
\usepackage{booktabs}
\usepackage{tikz-cd}
\usepackage{url}
\usepackage[pdftex, colorlinks=true,  citecolor=blue, linkcolor=blue, linktocpage=true, backref=page]{hyperref}
\usepackage[all]{xy}

\newcommand{\mydate}{
\number\day\space
\ifcase\month \or January\or February\or March\or April\or May\or June\or July\or August\or September\or October\or November\or December\fi 
\space\number\year}

\DeclareUrlCommand\arXiv{\urlstyle{same}}

\newtheorem{theorem}{Theorem}[section]
\newtheorem*{maintheorem}{Theorem}
\newtheorem{lemma}[theorem]{Lemma}
\newtheorem{proposition}[theorem]{Proposition}
\newtheorem{corollary}[theorem]{Corollary}

\theoremstyle{definition}

\newtheorem{remark}[theorem]{Remark}

\newtheorem*{acknowledgement}{Acknowledgement}

\theoremstyle{remark}

\newcommand{\PP}{\mathbb{P}}
\renewcommand{\AA}{\mathbb{A}}
\newcommand{\GG}{\mathbb{G}}

\newcommand{\ideala}{\mathfrak{a}}
\newcommand{\idealb}{\mathfrak{b}}

\newcommand{\shF}{\mathscr{F}}

\newcommand{\shI}{\mathscr{I}}

\newcommand{\catI}{\mathcal{I}}

\newcommand{\catN}{\mathcal{N}}
\newcommand{\catO}{\mathcal{O}}

\newcommand{\aff}{\text{\rm aff}}

\newcommand{\alg}{\text{\rm alg}}

\newcommand{\ant}{\text{\rm ant}}

\newcommand{\Ann}{\operatorname{Ann}}

\newcommand{\Aut}{\operatorname{Aut}}

\newcommand{\Bl}{\operatorname{Bl}}

\newcommand{\codim}{\operatorname{codim}}

\newcommand{\Der}{\operatorname{Der}}

\newcommand{\Diff}{\operatorname{Diff}}

\newcommand{\edim}{\operatorname{edim}}

\newcommand{\Fitt}{\operatorname{Fitt}}

\newcommand{\Gal}{\operatorname{Gal}}
\newcommand{\GL}{\operatorname{GL}}

\newcommand{\Hom}{\operatorname{Hom}}

\newcommand{\id}{{\operatorname{id}}}

\newcommand{\Kernel}{\operatorname{Ker}}

\newcommand{\Lie}{\operatorname{Lie}}

\newcommand{\lra}{\longrightarrow}

\newcommand{\Mat}{\operatorname{Mat}}

\newcommand{\maxid}{\mathfrak{m}}

\renewcommand{\O}{\mathscr{O}}

\newcommand{\pr}{\operatorname{pr}}
\newcommand{\Proj}{\operatorname{Proj}}

\newcommand{\quadand}{\quad\text{and}\quad}

\newcommand{\ra}{\rightarrow}

\newcommand{\reg}{\operatorname{reg}}

\newcommand{\sep}{{\operatorname{sep}}}

\newcommand{\Sch}{\text{\rm Sch}}

\newcommand{\Spec}{\operatorname{Spec}}
\newcommand{\Stab}{\operatorname{Stab}}

\newcommand{\Sym}{\operatorname{Sym}}

\newcommand{\uHom}{\underline{\operatorname{Hom}}}

\newcommand{\lieg}{\mathfrak{g}}

\begin{document}

\title[ ]
      {The inverse Galois problem for connected algebraic groups}.

\author[Michel Brion]{Michel Brion}
\address{Universit\'e Grenoble Alpes, Institut Fourier, CS 40700, 38058 Grenoble Cedex 9, France}
\curraddr{}
\email{michel.brion@univ-grenoble-alpes.fr}
 
\author[Stefan Schr\"oer]{Stefan Schr\"oer}
\address{Heinrich Heine University Düsseldorf, Faculty of Mathematics and Natural Sciences, Mathematical Institute, 40204 D\"usseldorf, Germany}
\curraddr{}
\email{schroeer@math.uni-duesseldorf.de}


\dedicatory{Revised version, 23 May 2024}

\begin{abstract}
We show that each connected group scheme of finite type over 
an arbitrary ground field is isomorphic to the component of the identity 
inside the automorphism group scheme of some projective, geometrically
integral scheme. The main ingredients are embeddings into
smooth group schemes, equivariant completions, blow-ups 
of orbit closures, Fitting ideals for K\"ahler differentials,
and Blanchard's Lemma.
\end{abstract}

\maketitle
\tableofcontents

\section*{Introduction}
\label{Introduction}

Let $k$ be a ground field of characteristic $p\geq 0$, and $X$ a proper scheme. 
According to a result of Matsumura and Oort  (\cite{Matsumura; Oort 1967}, 
Theorem 3.7),  the automorphism group scheme $\Aut_X$ exists and 
the connected component of the identity $\Aut^0_X$ is an \emph{algebraic group}, 
that is,  a group scheme of finite type.
Note that  the underlying scheme may be non-reduced for $p>0$.
Also note  that the group of connected components is not necessarily finitely generated 
(the first example was constructed by Lesieutre \cite{Lesieutre 2018}),
and understanding its structure attracted a lot of attention in the last decade.
 
Here we concentrate on the \emph{connected} algebraic group $\Aut_X^0$. 
Despite a very satisfactory structure theory for algebraic groups 
(see for example \cite{Milne 2017}),  surprisingly little is known about how the geometry 
of the scheme $X$ and the  geometry of the algebraic group $\Aut_X^0$ are related, 
in particular over imperfect ground fields and in presence of singularities.

The goal of this paper is to investigate the following problem, 
which can be seen as a higher-dimensional analog 
of the classical Inverse Galois Problem for number fields:
\emph{Given an algebraic group $G$, does there exist a proper, 
geometrically integral scheme $X$ whose automorphism group scheme 
is isomorphic to $G$?} 
 
Stated in this general form,  the answer to the problem is negative: 
Building on recent work of Lombardo and Maffei \cite{Lombardo; Maffei 2020}, 
Blanc and the first author \cite{Blanc; Brion 2021} showed that if $X$ is projective 
and $\Aut_X$ is an abelian variety, the automorphism group of this abelian variety 
must be finite (see \cite{Florence 2021} for further developments).
In particular, the selfproduct $E\times E$ of any elliptic curve  
does not occur as the automorphism group scheme of a projective, 
geometrically integral scheme.
  
The main result of this paper answers positively the above higher-dimensional
analog to the Inverse Galois Problem provided one restricts
attention to \emph{connected} algebraic groups:

\begin{maintheorem}
(See Thm.\ \ref{thm:main})
For every connected algebraic group $G$ over a ground field $k$, 
there is a projective,  geometrically integral scheme $X$ with $ \Aut_X^0= G$.
\end{maintheorem}
 
One may choose the $G$-action to be generically free. If $G$ is smooth 
of dimension $n\geq 1$, one may additionally achieve $X$ normal with 
$\dim(X)=\max(2n,3)$. 
The result generalizes \cite{Brion 2014}, Theorem 1,
where the case that $G$ is smooth and $k$ is perfect was treated.
As mentioned there, if $p = 0$ then one may additionally take $X$
to be smooth. This follows from the existence of a canonical 
resolution of singularities (see  e.g.~\cite{Kollar 2007}), 
which is not known to exist in positive characteristic and 
dimension at least $3$. So it is an open question whether
the above scheme $X$ may be chosen to be smooth. The recent
preprint \cite{Florence 2023} answers this question in the 
affirmative when $G$ is a \emph{linear} algebraic group,
not necessarily connected, and $p$ is arbitrary: 
then one may even achieve that $G$ is the full automorphism 
group scheme $\Aut_X$.

As an application of the main theorem,
we see that for any finite-dimensional restricted Lie algebra $\lieg$ 
there is a projective, geometrically integral scheme $X$ 
with $H^0(X,\Theta_X)=\lieg$, where $\Theta_X$ denotes the tangent sheaf. 
Recall that the Lie algebra $\lieg$ of an algebraic group in characteristic $p>0$
carries, beside the bracket $[x,y]$, the $p$-map $x^{[p]}$ as an additional structure,  and is called a \emph{restricted Lie algebra}.
Note that there are severe restrictions on the restricted Lie algebra
$H^0(X,\Theta_X)$ if the proper integral scheme $X$ is a surface of general type,
or more generally has \emph{foliation rank}  at most $1$, according to recent work
of Tziolas and the second author \cite{Schroeer; Tziolas 2023}.

We expect further applications of our result, for example in connection to classifying stacks $BG$, 
perhaps analogous to Totaro's approach \cite{Totaro 1999}, 
or with respect to non-abelian cohomology $H^1(k,G)$, as in \cite{Hilario; Schroeer 2023}, Section 9.

The idea for the proof of the above theorem  is to start with an inclusion $G\subset\Aut_Y$ 
for some  projective, geometrically integral scheme $Y$, chosen in such a way 
that $G$ is the stabilizer of some closed subscheme $Z\subset Y$, and     
pass to the blowing-up $X=\Bl_Z(Y)$. Combining naturality of blowing-ups 
with Blanchard's Lemma, we get inclusions $G\subset\Aut^0_X$ inside $\Aut^0_Y$,
and manage to infer equality by choosing the schemes $Y$ and $Z$  appropriately.
 
We actually proceed in two steps, treating the case that $G$ is smooth first. 
In this situation we regard $G$ as a homogeneous space for $G\times G$ 
via left and right multiplication, and define $Y=V\times V$  as the self-product
of some  $G \times G$-equivariant completion 
$G\subset V$. The center $Z\subset Y$ for the blowing-up 
stems from  certain graphs of morphisms $V\ra V$. 
For technical reasons, the case $\dim(G)=1$ has to be treated separately.
Along the way, we establish several facts on group scheme actions that appear to be of independent interest.

In a second step, we consider algebraic groups $G$ that are non-smooth, 
which happens only in characteristic $p>0$. 
To start with, we establish that $G$ embeds into some smooth connected 
algebraic group. Such an embedding is in no way canonical, and we loose 
control of dimensions. But together with the preceding paragraph we get  
$G\subset\Aut^0_Y$ for some projective, geometrically integral $Y$.   
Now the center $Z\subset Y$ is chosen as the  closure of some free $G$-orbit.
The local rings on such a free orbit are complete intersections, and acquire a rather simple description after base-changing to the algebraic closure of $k$.
We use this to get information on the \emph{Fitting ideals} of 
$\Omega^1_X$ on the blowing-up $X=\Bl_Z(Y)$, and the corresponding closed subschemes.
The latter are intrinsically attached to $X$, thus are stabilized by all group scheme actions,
and from  this we  infer $G=\Aut^0_X$. 
To understand these Fitting ideals, we have to make  explicit computations for  the blowing-up of the polynomial ring
$R=k[x_1,\ldots,x_n]$ with respect to ideals of the form $\ideala=(x_s^{\nu_s},\ldots,x_n^{\nu_n})$,
where the exponents are $p$-powers.  
 
Let us also discuss the famous Inverse Galois Problem, which can  be stated as follows:
Given a  number field $F$ and a finite group $G$, does there exist 
a finite Galois extension $F\subset E$ with automorphism group 
$\Aut(E/F)$ isomorphic to $G$? Hilbert showed this 
for the symmetric groups (\cite{Hilbert 1892}, page 124), and
Shafarevich established it for solvable groups 
(\cite{Safarevic 1954}, see also \cite{Neukirch; Schmidt; Wingberg 2008}). 
Despite many other positive results in this direction,  the general case remains open to date.
For overviews, see \cite{Serre 1992} or \cite{Voelklein 1996} or  \cite{Malle; Matzat 1999}. 
But note that  Fried and Koll\'ar \cite{Fried;  Kollar 1978} showed that the answer to the Inverse 
Galois Problem becomes affirmative if one drops the assumption that 
$F\subset E$ is Galois. This was extended to function fields in one variable 
in the recent preprint \cite{Bragg 2023}:  Given a finite \'etale group scheme $G$
and a (projective,  geometrically integral) regular curve $Y$ over $k$, 
there exists a regular curve $X$,  finite over $Y$,  such that $G = \Aut_{X/Y} = \Aut_X$.
In particular,  every finite \'etale group scheme is the full automorphism group scheme
of some regular curve.

\medskip
The paper is structured as follows:
In Section \ref{sec:prel} we fix notation and establish several useful general facts on group schemes.
Section \ref{sec:smooth} contains the proof of the main result in the case that $G$ is smooth.
In Section \ref{sec:fitting} we discuss Fitting ideals for K\"ahler differentials, and the stability of the resulting
closed subschemes with respect to group scheme actions.
Section \ref{sec:rees} contains some explicit computations of such Fitting ideals for certain Rees rings $R[\ideala T]$,
needed to understand blowing-ups of free $G$-orbits.
This is applied in Section \ref{sec:ns}, where we prove our main result for non-smooth $G$.

\begin{acknowledgement}
This research started when the first author visited the Heinrich Heine University D\"usseldorf,
and was continued during two visits of the second author at the University of Grenoble.
We both thank the host institutions for hospitality.
This research was also conducted in the framework of the   research training group
\emph{GRK 2240: Algebro-geometric Methods in Algebra, Arithmetic and Topology}, which is funded
by the Deutsche Forschungsgemeinschaft. 
Finally,  we thank the two referees for their careful reading
and helpful comments.
\end{acknowledgement}

\section{Preliminaries}
\label{sec:prel}

In this section we introduce notations and conventions, and establish
some general facts on group schemes. Throughout the paper, 
we work over a ground field $k$ of characteristic $p\geq 0$.
Choose an algebraic closure $k^{\alg}$,
and denote by $k^{\sep}$ the separable closure
of $k$ in $k^{\alg}$.
For any field extension $k\subset K$ and any  scheme $X$, we write 
$X_K=X \times_{\Spec(k)} \Spec(K)$ for the base-change.

Given a group scheme $G$, we denote the
neutral element by $e  \in G(k)$,
and the Lie algebra by $\Lie(G)$.
We will freely use the fact that every
subgroup scheme  $H\subset G$ is closed
(see \cite{SGA 3a}, Expos\'e $\text{VI}_\text{A}$,  Corollaire 0.5.2)
An \emph{algebraic group} is a group scheme
of finite type. A \emph{locally algebraic group} 
is a group scheme, locally of finite type. 
For any locally algebraic group $G$, the connected 
component of $e\in G$ is a connected algebraic subgroup, 
denoted by $G^0$ and called the \emph{neutral component} 
(see \cite {Demazure; Gabriel 1970}, Theorem II.5.1.1).
Note that $G$ may fail to be smooth. In this case we have  $p>0$,
and the local ring $\O_{G,e}$ is not regular.

For any proper scheme $X$, the automorphism group
functor $T \mapsto \Aut(X \times T/T)$ on the category $(\Sch/k)$ is represented
by a locally algebraic group $\Aut_X$, according to 
\cite{Matsumura; Oort 1967}, Theorem 3.7. Its Lie algebra is  
$\Der(\catO_X)=H^0(X,\Theta_X)$, the space of $k$-linear 
derivations $D :\O_X\ra \O_X$,
or equivalently the space of global sections of the tangent sheaf $\Theta_X=\uHom(\Omega^1_X,\O_X)$,
see loc.~cit., Lemma 3.4. 
Clearly, the neutral component $\Aut^0_X$ has the same Lie algebra.

Let $G$ be a group scheme. A $G$-\emph{scheme} 
is a scheme $X$ equipped with a $G$-action 
$a : G \times X \to X$.
As customary, we write $ g\cdot x= a(g,x) $  for  
$g \in G(T)$, $x \in X(T)$.

We will repeatedly use a result for $G$-actions called
\emph{Blanchard's Lemma}: Let $f: X \to Y$ be a proper morphism
of schemes of finite type and let $G$ be a connected algebraic
group acting on $X$. If the canonical map $\catO_Y \to f_*(\catO_X)$ 
is bijective, then there is a unique $G$-action on $Y$ such that
$f$ is equivariant (see \cite{Blanchard 1956}, Proposition 1.1 
for the original statement in the setting of complex geometry, 
and \cite{Brion 2017}, Theorem 7.2.1 for the above scheme-theoretic 
version). As a consequence, if in addition $X$ is proper 
then $f$ induces a homomorphism of group schemes 
$f_* : \Aut^0_X \to \Aut^0_Y$.

Given a group scheme $G$, a $G$-scheme $X$ and a closed subscheme 
$Z\subset X$, the subgroup functor comprising the $g\in G(T)$ 
such that $g|{T'} \cdot z\in Z(T')$
for all $T'\ra T$ and $z \in Z(T')$ is representable. 
The resulting subgroup scheme $\Stab_G(Z)\subset G$ is called 
the \emph{stabilizer}; it acts naturally on $Z$.
Likewise, the subgroup functor given by the condition $g|T'\cdot z=z$ 
is representable by a normal subgroup scheme of 
$\Stab_G(Z)$,  the \emph{kernel of its action on} $Z$
(see \cite{Demazure; Gabriel 1970}, Theorem II.1.3.6 for these facts).  In particular,  the kernel of the $G$-action on $X$ is
representable by a normal subgroup scheme.

If $X$ is proper and $Z$ is a closed subscheme, 
we write $\Aut_{X,Z}$ for the stabilizer of $Z$ in $G=\Aut_X$. 
Its Lie algebra is  the space of $k$-linear derivations $D:\O_X\ra \O_X$ 
that preserve the sheaf of ideals $\shI\subset\O_X$ corresponding to $Z$, see 
for example \cite{Martin 2022}, Section 2. The blow-up 
$\Bl_Z(X) = \Proj(\bigoplus_{i =0}^{\infty} \shI^i)$ is equipped with an action 
of $\Aut_{X,Z}$ such that the canonical morphism $f : \Bl_Z(X) \to X$ 
is equivariant (see loc. cit., Proposition 2.7).

A closed subscheme $Z\subset X$ is called $G$-\emph{stable}
if $\Stab_G(Z) = G$. Equivalently,
the morphism $a : G \times Z \to X$  factors through $Z$.
The latter notion of $G$-stability extends to 
arbitrary subschemes $Z\subset X$, which are intersections of closed subschemes with open sets.
Note that an  open set $U\subset X$ is
$G$-stable if and only if for all field
extension $k\subset K$ and  all $g \in G(K)$ and $x \in U(K)$,
we have $g \cdot x \in U(K)$.
 
We now collect auxiliary results on $G$-stability.
First recall that the \emph{schematic image} for a morphism   $f:X\ra Y$ of schemes
is the smallest closed subscheme $Z\subset Y$ over which the morphism factors  
(\cite{EGA I}, Section 6.10).
It   exists without further assumptions, 
and  corresponds to the largest quasicoherent sheaf of ideals $\shI\subset\O_Y$ with $f^{-1}(\shI)\O_X=0$.
If $f$ is quasicompact, the underlying topological space of $Z$
is the closure of the set-theoretic image, and $\shI$ is the kernel of $\O_Y\ra f_*(\O_X)$
(see \cite{Goertz; Wedhorn 2010}, Section 10.8 for these facts).
Note that if $f$ is also quasiseparated,  the direct image $f_*(\O_X)$ is already quasicoherent.
If $f$ is proper, the schematic image coincides with the set-theoretical image,  endowed with some scheme structure.
The following observation is a version of \cite{Martin 2022}, Lemma 2.5.

\begin{lemma}
\label{lem:image}
Let $G$ be a group scheme, and $f : X \to Y$ be an equivariant quasicompact morphism of $G$-schemes.
Then the schematic image $Z\subset Y$ is $G$-stable. 
\end{lemma}

\begin{proof}
By assumption, we have a commutative square
\[ \xymatrix{
G \times X 
\ar[r]^-{\id \times f}\ar[d]_{a} 
& G \times Y \ar[d]^{b} \\ 
X \ar[r]^-{f} & Y, \\}
\]
where $a$ and $b$ correspond to the $G$-actions.
Also, the schematic image of $a$ is the whole $X$,
since $e\cdot x = x$.
In view of the transitivity of schematic images 
(see \cite{EGA I}, Proposition 6.10.3), it suffices 
to show that the schematic image of $\id \times f$
is $G \times Z$. 

Since $f$ is quasicompact, the formation of 
its schematic image commutes with flat base change 
(see \cite{EGA IV.3}, Th\'eor\`eme 11.10.5). 
This yields the desired assertion by using 
the cartesian square
\[ \xymatrix{
G \times X 
\ar[r]^-{\id \times f}\ar[d]_{\pr_X} 
& G \times Y \ar[d]^{\pr_Y} \\ 
X \ar[r]^-{f} & Y, \\}
\]
where the projections $\pr_X$ and $\pr_Y$ are flat. 
\end{proof}

\begin{lemma}\label{lem:smooth}
Let $G$ be a group scheme, and $X$ a $G$-scheme of finite type. 
Then the smooth locus $U\subset X$ is $G$-stable.
\end{lemma}

\begin{proof}
Since $U\subset X$ is open, it suffices 
to check that for all field extensions $k\subset K$ and 
all $g \in G(K)$, $x \in U(K)$, we have $g \cdot x \in U(K)$.
As the formation of the smooth locus commutes 
with base change by field extensions (see 
\cite{EGA IV.4}, Proposition 17.3.3, 
Proposition 17.7.1), it suffices to treat the case that $K=k$, 
and that this field is algebraically closed. 
Then $U$ comprises the points where the local rings 
are regular, and the statement is immediate.
\end{proof}

\begin{lemma}\label{lem:irr}
Let $G$ be a connected algebraic group
and $X$ a $G$-scheme of finite type. 
Then every connected component of $X$
is $G$-stable. 
If in addition $X$ is geometrically reduced, 
then every irreducible component of $X$ is 
$G$-stable.
\end{lemma}

\begin{proof}
Let $Y\subset X$ be a connected component.
Then the set-theoretical image  $G \cdot Y = a(G \times Y)$
is open in $X$, since the   morphism $a:G\times X\ra X$ describing the action 
is flat,  and $G \times Y$ is
of finite type. The latter
is connected, since $G$ is geometrically
connected and $Y$ is connected. 
Thus, $G \cdot Y$ is connected as well.
As $Y$ is contained in $G \cdot Y$, we have 
$G \cdot Y = Y$ and hence $Y$ is $G$-stable.

Now assume that $X$ is geometrically reduced. Let   $X_1,\ldots,X_n$ be
the irreducible components, endowed with reduced scheme structures, and $U\subset X$ the smooth locus.
Then  $U_i=U\cap X_i$ are the connected components of $U$,
and $U_i\subset X_i$ is schematically dense.
In view of Lemma \ref{lem:smooth} and the preceding paragraph, it follows that every $U_i$ is 
$G$-stable. Thus, each $X_i$ is $G$-stable 
by Lemma \ref{lem:image}.
\end{proof}

Given a group scheme $G$ and a $G$-scheme $X$,
we denote by $X^G\subset X$ the \emph{scheme of fixed points}.
This is the closed subscheme  representing 
the subfunctor comprising the $x\in X(T)$
such that for all $T'\ra T$ and $g\in G(T')$ we have
$g\cdot x=x$, see for example \cite{Demazure; Gabriel 1970},
Theorem II.1.3.6. The arguments in loc.\ cit.\ reveal that 
for any closed subscheme $F\subset G$, 
the \emph{scheme of $F$-fixed points} $X^F\subset X$ 
is indeed representable by a closed subscheme.

\begin{lemma}
\label{lem:fixed}
Let $G$ be a locally algebraic group, and $X$ 
a $G$-scheme of finite type. Then there exists 
a finite closed subscheme $F\subset G$ such that 
$X^G = X^F$. If $G$ is smooth,  then $F$ may 
be taken \'etale.
\end{lemma}

\begin{proof}
Consider the family of closed subschemes
$X^F$ of $X$, where $F$ runs over the finite closed
subschemes of $G$. Since $X$ is noetherian, we may choose such a
subscheme $F_0$ with $X^{F_0}$ minimal.
For each $F$ containing $F_0$ we have 
$X^F = X^{F_0}$ by minimality.
Equivalently, $F$ is a subscheme of the 
largest subgroup scheme $H \subset G$
that acts trivially on $X^{F_0}$. 
Moreover, the family of finite subschemes of $G$
containing $F_0$ is schematically dense in $G$, 
since the latter is locally of finite type. 
So $H = G$ and hence $X^G = X^{F_0}$ as desired.

If $G$ is smooth, then we argue similarly 
by considering finite closed subschemes that are \'etale; these 
also form a schematically dense family, since
$G(k^{\sep})$ is dense in $G_{k^{\sep}}$.
\end{proof}

\begin{remark}
\label{rem:fixed}
In characteristic zero, every algebraic group $G$ is generated 
by some finite \'etale subscheme $F$ (see \cite{Brion 2014}, 
Lemma 3). Hence $X^G = X^F$ for any $G$-scheme of finite type $X$. 
But this does not hold true in characteristic $p > 0$: 
For example, the additive group $\GG_a$ is not 
generated by any finite \'etale subscheme $F$, because  
the finite set $F(k^{\alg})$ generates a finite subgroup 
of $\GG_a(k^{\alg}) = k^{\alg}$. Moreover, one may check
that there exists no finite \'etale subscheme $F\subset \GG_a$ 
such that $X^{\GG_a} = X^F$ for all $\GG_a$-schemes $X$ 
of finite type.
\end{remark}

We will also need a version of a classical criterion
for a homomorphism of algebraic groups to be an 
isomorphism.

\begin{lemma}
\label{lem:hom}
Let  $f:G\ra H$ be a homomorphism of locally algebraic groups.
Assume that the induced maps
$ G(k^{\alg}) \to H(k^{\alg})$
and $ \Lie(G) \to \Lie(H)$ 
are bijective, and that $G$ is smooth. Then $f$ is an isomorphism.
\end{lemma}

\begin{proof}
By descent, it suffices to treat the case that  $k=k^\alg$.
The kernel $N$ of $f$ is a locally algebraic group that satisfies 
$N(k) = \{ e \}$ and $\Lie(N) = \{ 0 \}$,
in view of our assumptions. Thus $N$ is trivial.

Recall that $G^0$ and $H^0$ are algebraic groups; moreover,
$f$ restricts to a homomorphism $f^0 : G^0 \to H^0$. 
Applying \cite{Demazure;  Gabriel 1970},
Corollaire II.5.5.5, we see that $f^0$ is an open immersion,
and hence an isomorphism.
It remains to verify that the induced homomorphism 
$\pi_0(f) : \pi_0(G) = G/G^0\ra H/H^0 = \pi_0(H)$
between \'etale  group schemes is an isomorphism. 
These group schemes are actually constant, because $k$
is separably closed. Furthermore, $\pi_0(f)$ is surjective, 
because $G\ra H$ and $H\ra H/H^0$
are surjective on $k$-rational points.
Suppose the kernel of $\pi_0(f)$
contains some  non-trivial $\bar{g}\in G/G^0$.
Represent it by some $g\in G(k)$. Then $g\not\in G^0(k)$ and 
$f(g)\in H^0 = f(G^0)$. Thus, there exists $g' \in G^0(k)$
such that $f(g) = f(g')$. Then $g^{-1}\cdot g'$ is 
non-trivial and belongs to $N$, contradiction.
\end{proof}

The following observation will be a key ingredient in 
the proof of our main theorem in positive characteristics:

\begin{proposition}
\label{prop:subgroup}
Each connected algebraic group $G$ is isomorphic to a subgroup scheme 
of a smooth connected algebraic group $H$.
\end{proposition}

\begin{proof}
There exists a connected 
affine normal subgroup scheme $N$ such that $G/N$ 
is an abelian variety (see for example \cite{Milne 2017}, 
Theorem 8.28). 
Also, the affinization 
morphism $G \to G^{\aff} = \Spec H^0(G,\O_G)$ 
is a faithfully flat homomorphism of algebraic groups. 
Moreover, its kernel $G_{\ant}$ is smooth, connected
and contained in the center  (see 
\cite{Demazure; Gabriel 1970},
Th\'eor\`eme III.3.8.2 and 
Corollaire III.3.8.3, or \cite{Brion 2017}, Theorem 1).

We have $G = G_{\ant} N$.
Indeed, $G_{\ant} N$ is a normal subgroup scheme
of $G$, and the quotient $G/G_{\ant} N$
is both affine (as a quotient of $G/G_{\ant}$)
and an abelian variety (as a quotient of $G/N$).
So we obtain an exact sequence 
\[ 1 \longrightarrow G_{\ant} \cap N
\longrightarrow G_{\ant} \times N 
\longrightarrow G \longrightarrow 1, \]
where $G_{\ant} \cap N$ is identified with a
central subgroup scheme of $G_{\ant} \times N$.

We may identify the affine algebraic group $N$ with a subgroup scheme
of some general linear group $\GL_m$. By construction, $N$ is contained 
in the centralizer 
$C = C_{\GL_m}(G_{\ant} \cap N)$.
Moreover, $C$ is smooth (see 
\cite{Herpel 2013}, Lemma 3.5,
based on \cite{Demazure; Gabriel 1970},
Proposition II.2.1.6).
Hence $G$ is isomorphic to a subgroup scheme
of the quotient 
$H=(G_{\ant} \times C)/(G_{\ant} \cap N)$,
which is smooth as well. Since $G$ is connected, 
it is already contained in the neutral component $H^0$.
\end{proof}

\begin{remark}
\label{rem:subgroup}
Each algebraic group $G$ is an extension of the finite \'etale group $G/G^0$
by the connected algebraic group $G^0$, but the projection $G\ra G/G^0$ 
usually has no section. By the above, we may embed $G^0$ into some
smooth algebraic group, but we do not know how to carry this over to $G$.

However, it is easy to construct
examples of   algebraic groups $G$  that cannot 
be embedded into  a connected algebraic group: Let $E$ be an elliptic 
curve and consider the semidirect product $G = E \rtimes \{ \pm 1 \}$, 
where $\{ \pm 1 \}$ denotes the constant group 
scheme of order two acting on $E$ via $x \mapsto \pm x$.  
Then $G$ is a smooth disconnected algebraic group. 
Suppose $G$ is a subgroup scheme of some 
connected algebraic group $H$, then $E$ is central 
in $H$ (see \cite{Milne 2017}, Corollary 8.13). 
But $E$ is not central in $H$, a contradiction.
\end{remark}

\section{The main result  in the  smooth case}
\label{sec:smooth}

We now  formulate  the main result of this paper:

\begin{theorem}
\label{thm:main}
For any connected algebraic group $G$ over a field $k$, there is a projective, 
geometrically integral $G$-scheme $X$ with $\Aut_X^0= G$, and generically free 
action.  If $G$ is smooth of dimension $n\geq 0$, then one may choose  
$X$ normal with $\dim(X)=\max(2n,3)$.
\end{theorem}

Let us give an immediate application in characteristic $p>0$, which was actually the starting point for our research:
Recall that for any group scheme $G$ and any scheme $X$
the Lie algebras $\Lie(G)$ and $H^0(X,\Theta_X) = \Der(\catO_X)$ 
carry as additional structure the so-called $p$-map 
$D\mapsto D^{[p]}$. For the Lie algebra of global vector fields, 
this is just the $p$-fold composition in the  algebra 
$\Diff(\catO_X)$ of $k$-linear differential operators.
This leads to the notion of \emph{restricted Lie algebra}. 
By \cite{Demazure; Gabriel 1970}, Proposition II.7.4.1,  
the functor $G\mapsto \Lie(G)$ is an equivalence 
between the category of algebraic groups annihilated
by the relative Frobenius, and the category of finite-dimensional 
restricted Lie algebras. Such algebraic groups contain but one point. As an immediate
consequence of the theorem we get:

\begin{corollary}
\label{cor:lie algebra}
For every finite-dimensional restricted Lie algebra $\lieg$ over our ground field $k$ of characteristic $p>0$,
there is a  projective,  geometrically integral scheme $X$ such that the restricted Lie algebra $H^0(X,\Theta_X)$
is isomorphic to $\lieg$.
\end{corollary}

This is quite different from the case of characteristic zero,
where the Lie algebras of global vector fields on proper schemes
are exactly the finite-dimensional linear Lie algebras
(see \cite{Brion 2014}, Corollary 1).

Let us now turn to the proof of Theorem \ref{thm:main}.
This section contains the arguments  under the additional assumption 
that $G$ is smooth. The rough idea is as follows:
In a first step, we realize $G$ 
as the group scheme of those automorphisms of 
a normal, projective, geometrically integral 
scheme $V$ that  commute with finitely many 
automorphisms $f_1,\ldots, f_m$.
In a second step, we   replace 
this commutator condition with a stabilizer condition 
for  a closed subscheme;
for this, we replace $V$ with $Y = V \times V$, with
closed subscheme the union $Z$ of the graphs
$\Gamma_{f_i}$. 
In a third step, we consider the normalized blow-up $X\ra Y$ 
with center $Z$, and check that $G= \Aut^0_X$ provided 
that $n \geq 2$. The case $n=1$ requires some additional 
attention, and is reduced to $n\geq 2$  via a small trick. 

\bigskip\noindent 
{\bf Step 1.}
The product group scheme $G\times G$ acts from the left 
on the scheme $G$ via the formula $(g_1,g_2)\cdot g=g_1gg_2^{-1}$.
This action is transitive, and the stabilizer of $e\in G$ 
is the diagonal $\Delta(G)$; we may thus identify $G$ with
the homogeneous space $(G \times G)/\Delta(G)$. 
According to \cite{Milne 2017}, Theorem 8.44, $G$ admits a 
$G \times G$-\emph{equivariant projective completion} $V$.
In other words, $V$ is a projective $G\times G$-scheme, 
containing $G$ as an open set that is schematically dense and  
$G\times G$-stable. 
Clearly, $V$ is geometrically integral,
since $G$ is smooth and geometrically irreducible 
(\cite{SGA 3a}, Expos\'e $\text{VI}_{A}$, Proposition 2.4).
Also, the $G \times G$-action on $V$ lifts uniquely to the normalization 
of $V$,  by \cite{Brion 2017}, Proposition 2.5.1. 
Thus, we may also assume that $V$ is normal.

Denote by $G \times e$ and $e\times G$ 
the respective left and right copies of $G$ inside $G\times G$. 
These subgroup schemes commute with each other. This gives 
a homomorphism of group schemes
$$
\psi: G = G \times e \longrightarrow 
\Aut_V^{e \times G},
$$
where $\Aut_V^{e \times G}$ denotes the
centralizer of $e \times G\subset \Aut_V$. This is indeed the 
scheme of fixed points for the
conjugacy action of $e \times G$ on $\Aut_V$.

\begin{lemma}\label{lem:psi}
In the above situation,  $\psi$ is an isomorphism.
\end{lemma}

\begin{proof}
It suffices to treat the case that $k$ is algebraically closed.
Observe that the kernel of the  
$G \times G$-action on $G$ equals the diagonal of the center $Z(G)$. Thus 
this is also the kernel of the 
$G \times G$-action on $V$, i.e., of the
corresponding homomorphism  $G \times G \ra \Aut_V$. 
Since  the intersection of $G\times e$ with the diagonal of the center is trivial, the homomorphism
$\psi$ has a trivial kernel, and hence is a closed immersion 
(as follows from \cite{Demazure; Gabriel 1970}, Proposition II.5.5.1).
In order to apply Lemma \ref{lem:hom}, we have to show that $\psi$
induces a surjection on rational points and Lie algebras.

Let $h \in \Aut^{e \times G}_V(k)$. Then $h:V\ra V$  stabilizes $G\subset V$,
which is the open orbit of the action of $e \times G$.
Since  the induced  $h:G\ra G$
commutes with the $e\times G$-action by multiplication from the right, it must be 
the left multiplication by some $g \in G(k)$. Since $G(k)\subset V$
is schematically dense,
we obtain $h = \psi(g)$. Thus $\psi$ is surjective on rational points. 

It remains to check that the injective map 
$
\Lie(\psi) : \Lie(G) \ra \Lie(\Aut^{e \times G}_V) 
$ is bijective. The term on the right is  $\Der^{e \times G}(\O_V)$
by \cite{Demazure; Gabriel 1970}, Proposition II.4.2.5. 
Let $i:G\ra V$ be the inclusion map of the open set $G$, which is schematically dense
and $G\times G$-equivariant. Then $i$ induces an injection of Lie algebras
$$
i^* : \Der^{e \times G}(\O_V) 
\longrightarrow
\Der^{e \times G}(\O_G) = \Lie(G),
$$
such that $i^* \circ \Lie(\psi) = \id_{\Lie(G)}$.
Having an injective left inverse, the injective linear map 
$\Lie(\psi)$ must be bijective.
\end{proof}

In particular, the homomorphism $\psi$ gives an identification 
$G=\Aut_V^{0, e \times G}$,
the centralizer of $e \times G$ in $\Aut^0_V$.

\medskip\noindent
{\bf Step 2.}
We now apply Lemma \ref{lem:fixed} to the conjugacy action of $e \times G$
on $\Aut^0_V$, and find some finite \'etale closed subscheme $F\subset G$
such that $G=\Aut_V^{0,e\times F}$.
Clearly, we may assume in addition that $e \in F(k)$. 
Consider the graph morphism
\[ \gamma : F \times V \longrightarrow 
V \times V, \quad
(g,v) \longmapsto (v, (e,g) \cdot v). \]
This morphism is finite: Obviously the morphism 
$ F\times V\ra V$ given by $(g,v)\mapsto v$ is finite. 
The same holds for $(g,v)\mapsto (e,g)\cdot v$, because 
it is isomorphic to the projection
via $(g,v)\mapsto (g,(e,g)^{-1}\cdot v)$.
Using that products and compositions of finite morphisms are finite, 
and that $F\times V$ is separated,
we see that the graph morphism is finite.

We set for simplicity $Y = V \times V$, 
and denote by $Z \subset Y$ the schematic image 
of $\gamma$. 
We now collect some observations and results
on the geometry of $Y$ and $Z$. 
Since $V$ is projective and geometrically 
integral, we see that $Y$ is projective and 
geometrically integral as well. Moreover, 
we have $\dim(Y) = 2 \dim(V) = 2 \dim(G) = 2n$. 
Also, $Z$ is geometrically reduced and 
equidimensional of dimension $n$.

Note that $V$ is not necessarily geometrically
normal (this happens if $G$ is a non-rational form
of the additive group, see \cite{Russell 1970}). 
Still, we have: 

\begin{lemma}\label{lem:nor}
The scheme $Y$ is normal.
\end{lemma}

\begin{proof}
We use Serre's Criterion. Since $V$ satisfies Serre's Condition $(S_2)$, 
the same holds for $Y$. It remains to verify that $Y$ is regular in codimension one, 
in other words,  satisfies  the condition $(R_1)$. Since this holds for $V$,
the closed set $Y\smallsetminus V_{\reg} \times V_{\reg}$ has codimension at least two,
where $V_{\reg}$ denotes the regular locus. 
Thus, it suffices in turn to show that
$V_{\reg} \times V_{\reg}$ satisfies $(R_1)$.
As $G$ is smooth, $G \times V_{\reg}$ is regular and therefore
$$
(V_{\reg} \times G) \cup (G \times V_{\reg}) \subset (V \times V)_{\reg}.
$$
Moreover, the complementary closed set
has codimension at least two, because 
\[ (V_{\reg} \times V_{\reg}) \smallsetminus 
((V_{\reg} \times G) \cup (G \times V_{\reg}))
= (V_{\reg} \smallsetminus G) \times 
(V_{\reg} \smallsetminus G). \]
Thus, $V_{\reg} \times V_{\reg}$ indeed satisfies $(R_1)$.
\end{proof}

By Galois descent, we may identify the 
finite \'etale subscheme $F$ of $G$ with 
the finite subset 
$F(k^{\sep}) \subset G(k^{\sep})$,
equipped with the  action of the absolute
Galois group $\Gal(k^\sep/k)$. The scheme 
$Z_{k^{\sep}}$ is the union of the graphs 
$\Gamma_f$, where $f \in F(k^{\sep})$ and 
$\Gamma_f$ denotes the schematic image of 
the closed immersion
\[ V_{k^{\sep}} \longrightarrow Y_{k^{\sep}},
\quad v \longmapsto (v,(e,f) \cdot v).  \]
Moreover, each $\Gamma_f$ is an irreducible
component of $Z_{k^{\sep}}$. Also, note that
$\Gamma_e$ is the diagonal $\Delta(V)$. 
The action of $G_{k^{\sep}} \times e$ on 
$Y_{k^{\sep}}$ 
stabilizes each graph $\Gamma_f$, and hence
$Z$ is $G \times e$-stable. 

Let $Y_0 \subset G \times G$ be the
open set such that $Y_{0,k^{\sep}}$ consists 
of the points which lie in at most one graph 
$\Gamma_f$, where $f \in F(k^{\sep})$. Clearly,
$Y_0$ is smooth, dense and $G \times e$-stable;
moreover, $G \times e$ acts freely on $Y_0$.
In particular, the $G \times e$-action on $Y$
is generically free. Likewise, 
$Z_0 = Z \cap Y_0$ is an open set of $Z$ which
is smooth, dense and $G\times e$-stable.

As a consequence of Blanchard's Lemma, the canonical 
homomorphism 
\[ \Aut^0_V \times \Aut^0_V 
\longrightarrow \Aut^0_Y, 
\quad (g,h) \longmapsto g \times h \]
of algebraic groups is an isomorphism
(see \cite{Brion 2017}, Corollary 7.2.3 for details). 
Throughout, we regard it as an identification.
Also, note that the diagonal homomorphism
\[ \Delta : 
\Aut^0_V \longrightarrow \Aut^0_Y, 
 \quad g \longmapsto  g \times g \]
 of algebraic groups is a closed immersion.

\begin{lemma}\label{lem:diag} 
With the above notation, we have
$\Stab_{\Aut^0_Y}(\Gamma_e) = \Delta(\Aut^0_V)$.
Moreover, 
$$
\Stab_{\Aut^0_{Y_{k^{\sep}}}}(\Gamma_e) \cap
\Stab_{\Aut^0_{Y_{k^{\sep}}}}(\Gamma_f)
= \Delta(\Aut^{0,e \times f}_{V_{k^{\sep}}})
$$
for every $f \in F(k^{\sep})$.
\end{lemma}

\begin{proof}
Let $T$ be a scheme and   
$g,h \in \Aut(V \times T/T)$. 
Then  the resulting 
$g \times h \in  \Aut(Y\times T/T)$ stabilizes
$\Delta(V) \times T$ if and only if
for any $T'\ra T$ and $v \in V(T')$, 
we have 
$(g \times h)(v,v) \in \Delta(V)(T')$, in other words $g(v) = h(v)$. This yields the first
equality. The second equality is checked
similarly.
\end{proof}
 
Next, recall the identification $G \times e = \Aut_V^{0,e \times G}$;
we denote by $\eta : G \ra \Aut_V^0$ the corresponding
closed immersion.

\begin{lemma}\label{lem:psipsi}
The homomorphism of algebraic groups
\[ \eta \times \eta : G \longrightarrow 
\Aut^0_V \times \Aut^0_V = \Aut^0_Y, \quad
g \longmapsto ((v,w) \mapsto ((g,e) \cdot v, (g,e) \cdot w)))\]
yields an identification 
$G = \Aut_{Y,Z}^0$.
\end{lemma}

\begin{proof}
We may assume that $k$ is algebraically closed.  
Since the kernel of $\eta$ is trivial,
the same holds for  $\eta \times \eta$, so the latter is a closed immersion.
Let $f \in F$. Then $\Gamma_f$ is stable
by the action of $G$ on $Y$ via $\psi\times\psi$,
since $G \times e$ centralizes $e \times f$. 
Thus, the $G$-action on $Y$  stabilizes $Z$, and $\psi\times\psi$ factors over 
$\Aut_{Y,Z}^0$.

It remains to check that the resulting closed immersion 
$G\ra\Aut_{Y,Z}^0$ is an isomorphism.
The  group scheme $ \Aut_{Y,Z}^0$ stabilizes the 
irreducible component $\Gamma_f\subset Z$, $f\in F$ in view of 
Lemma \ref{lem:irr}. In other words, $\Aut_{Y,Z}^0$ stabilizes 
the $\Gamma_f\subset Y$.
By Lemma \ref{lem:diag}, it is contained in the intersection of the 
$\Delta(\Aut_V^{0,e\times f})$. By construction, we have 
$G=\Aut_V^{0,e \times F}=\bigcap_{f\in F} \Aut_V^{0,e\times f}$, 
and the assertion follows. 
\end{proof}

\noindent
{\bf Step 3.}
Let $X$ be the normalization of the blow-up $\Bl_Z(Y)$  
with respect to the center $Z\subset Y$, and write
$f:X\ra Y$ for the resulting  birational morphism.
Then $X$ is again normal, projective,
and geometrically integral. Moreover,
the $G$-action on $Y$ via $\eta \times \eta$
lifts uniquely to an action on $X$ via
a homomorphism
\[ f^* : G \longrightarrow \Aut^0_X. \]

\begin{lemma}\label{lem:f}
If $n \geq 2$ then the above map
$f^*$ is an isomorphism.
\end{lemma}

\begin{proof}
We begin with some observations. We have
$\catO_Y \stackrel{\sim}{\longrightarrow}
f_*(\catO_X)$ by the normality of $Y$ (Lemma 
\ref{lem:nor}) and Zariski's Main Theorem. 
Thus, Blanchard's Lemma 
yields a homomorphism of algebraic groups 
\[ f_* : \Aut^0_X \longrightarrow \Aut^0_Y. \]
Since $f$ is birational, $f^*$ and $f_*$
have trivial kernels, and hence are closed
immersions. Moreover, the composition
$f_* \circ f^* : G \to \Aut^0_Y$
equals $\eta \times \eta$, since this holds
over a schematically dense open set of $Y$. 
Using Lemma \ref{lem:psipsi}, 
we may thus identify $G$ with $\Aut^0_{Y,Z}$, 
and $f_* \circ f^*$ with $\id$.
Also, note that the pull-back morphism
\[ f_0 : X_0 = f^{-1}(Y_0) \longrightarrow Y_0 \] 
is the blow-up of the smooth variety $Y_0$
along the smooth subvariety $Z_0$ of pure
codimension $n$. In particular, $X_0$
and the exceptional divisor $E_0$ of $f_0$
are smooth. Thus, the schematic closure 
$E \subset X$ of $E_0$ is geometrically reduced. 
Moreover, the schematic image of $E$ in $Y$ is $Z$, 
since $f_0(E_0) = Z_0$ is schematically dense in $Z$ 
(but $E$ may be strictly contained in the exceptional 
locus of $f$).

We now show that the closed immersion $f^*$ 
is an isomorphism, by checking that it is 
surjective on $k^{\alg}$-points and on Lie
algebras (Lemma \ref{lem:hom}). We may thus
assume $k$ algebraically closed. Then 
$\Aut^0_X(k)$ stabilizes the exceptional locus 
of $f$, and hence its irreducible components 
(Lemma \ref{lem:irr}). Thus, $\Aut^0_X(k)$
stabilizes $E$, and hence $Z$. This shows 
that $f^*$ is surjective on $k$-points.

Next, we show that
$\Lie(f^*) : \Lie(G) \to \Lie(\Aut^0_X)$
is surjective. For this, we use a rigidity
property of the exceptional divisor $E_0$.
We have an exact sequence of coherent sheaves
on $X_0$:
\[ 0 \longrightarrow \catO_{X_0}
\longrightarrow \catO_{X_0}(E_0)
\longrightarrow \catO_{E_0}(E_0)
\longrightarrow 0, \]
and hence an exact sequence on $Y_0$
\[ 0 \longrightarrow f_{0,*}(\catO_{X_0})
\longrightarrow f_{0,*}(\catO_{X_0}(E_0))
\longrightarrow f_{0,*}(\catO_{E_0}(E_0))
\longrightarrow R^1f_{0,*}(\catO_{X_0}). 
\]
Moreover, 
$f_{0,*}(\catO_{X_0}) = \catO_{Y_0}$
by Zariski's Main Theorem again; 
$f_{0,*}(\catO_{X_0}(E_0)) = \catO_{Y_0}$
as $\codim_{Y_0} f_0(E_0) = n \geq 2$;
and $R^1f_{0,*}(\catO_{X_0}) = 0$
by \cite{SGA 6}, Expos\'e VII, Lemme 3.5. 
Thus, we have
$f_{0,*}(\catO_{E_0}(E_0)) = 0$. In particular,
$H^0(E_0, \catO_{E_0}(E_0)) = 0$.

Now let $D \in \Lie(\Aut^0_X) = \Der(\catO_X)$. 
Then $D$ induces
$D_0 \in \Der(\catO_{X_0}) = H^0(X_0,\Theta_{X_0})$.
Let $\catN_{E_0} = \catO_{E_0}(E_0)$, then 
the image of $D_0$ under the composition
$\Theta_{X_0} \to \Theta_{X_0}\vert_{E_0}
\to \catN_{E_0}$ is $0$.
As 
$H^0(E_0,\catN_{E_0}) = \Hom_{X_0}
(\catI_{E_0},\catO_{X_0}/\catI_{E_0})$
where $\catI_{E_0} = \catO_{E_0}(-E_0)$,
this means that
$D_0 \in \Der(\catO_{X_0}; \catI_{E_0})$
(the derivations of $\catO_{X_0}$ which preserve
the sheaf of ideals of $E_0$).
In view of Lemma \ref{lem:image}, it follows that
$D \in \Der(\catO_X; \catI_E)$.
So the image of $D$ in $\Der(\catO_Y)$ lies
in $\Der(\catO_Y;\catI_Z)$ (as follows e.g. from
Lemma \ref{lem:image} again). Since 
$\Der(\catO_Y;\catI_Z) = \Lie(G)$ by 
Lemma \ref{lem:psipsi}, this completes the proof.
\end{proof}

It remains to treat the case that $n = 1$.
Then $Y$ is a surface, and $Z$ a curve.
Choose an elliptic curve $E$ with origin $0$.
Then $\Aut^0_{E,0}$ is trivial, and hence we obtain
$\Aut^0_{Y,Z} \stackrel{\sim}{\to}
\Aut^0_{Y \times E,Z \times 0}$ by Blanchard's
Lemma. In view of Lemma \ref{lem:psipsi},
we thus have 
$G \stackrel{\sim}{\to} 
\Aut^0_{Y \times E,Z \times 0}$,
where $Y \times E$ is a threefold, and 
$Z \times 0$ a curve. We now consider the
normalized blow-up $X$ of $Y \times E$ along 
$Z \times 0$, and argue as in the proof of 
Lemma \ref{lem:f} to get
$G \simeq \Aut^0_X$.

\begin{remark}
\label{rem:dim}
Let $G$ be a smooth connected algebraic group of
dimension $1$. If the ground field $k$ is perfect, 
then $G$ is either an elliptic curve, or 
the additive group $\GG_a$, or a $k$-form of 
the multiplicative group $\GG_m$. In each of these 
cases, one may easily construct a smooth projective 
surface $X$ such that $G \simeq \Aut^0_X$ 
(see the end of Section 2.1 in \cite{Brion 2014} 
for details). As a consequence, Theorem 
\ref{thm:main} holds with $\dim(X) = 2 n$.

If $k$ is imperfect, then we get in addition
the non-trivial $k$-forms of $\GG_a$; 
these are described in \cite{Russell 1970}. 
Given such a $k$-form $G$, 
is there a normal projective surface $X$
such that $G \simeq \Aut^0_X$? If the answer
is affirmative, then Theorem \ref{thm:main}
holds again with $\dim(X) = 2n$. We have 
just seen that there is a normal projective 
threefold $X$ with $G \simeq \Aut^0_X$, but 
one may check that there is no normal projective 
curve satisfying this assertion.
\end{remark}

\section{Fitting ideals for K\"ahler differentials}
\label{sec:fitting}

One crucial idea to prove Theorem \ref{thm:main} for non-smooth  algebraic groups  is to control actions via
Fitting ideals for K\"ahler differentials.  In this section we collect some more or less obvious facts, which
turn out to be very useful. 
Let us start by recalling the theory of Fitting ideals (compare \cite{Fitting 1936} and \cite{Eisenbud 1995}, Section 20.2): Suppose $M$ is a module of finite presentation
over some commutative ring $R$, and choose a particular presentation
$$
R^{\oplus n}\stackrel{P}{\lra} R^{\oplus m}\lra M\lra 0.
$$
Let  $\ideala_i\subset R$ be the ideal generated by the $(m-i)$-minors of the matrix $P\in\Mat_{m\times n}(R)$, for any integer $i\leq m$.
Then $\ideala_m$ is the unit ideal,  $\ideala_{m-1}$ is generated by the matrix entries, and we have $\ideala_i\supset\ideala_{i-1}$ according to  Laplace expansion
for determinants.  It turns out that these ideals depend only on the module rather than the presentation,
and one calls them \emph{Fitting ideals} $\Fitt_i(M)=\ideala_i$. 
Note that $\Fitt_m(M)=R$, and our definition   extends to all integers   by setting $\Fitt_i(M)=R$ for $i>m$.
One immediately sees that $\Fitt_{i+1}(M\oplus R ) =\Fitt_i(M)$.
Also note that the \emph{annihilator ideal} $\idealb=\Ann_R(M)$ is related to   Fitting ideals by the formulas
$\idealb \cdot \Fitt_{i+1}(M) \subset \Fitt_i(M)$ and 
$\idealb^m \subset\Fitt_0(M) \subset \idealb$.

Obviously Fitting ideals commute with arbitrary base change, 
in particular with localizations. This leads to the following generalization:
Let $\shF$ be a quasicoherent sheaf of finite presentation on some scheme $X$.
Then the Fitting ideals for $M=\Gamma(U,\shF)$ for the various affine open sets $U=\Spec(R)$ are compatible,  
thus define a quasicoherent \emph{sheaf of Fitting ideals} $\Fitt_i(\shF)\subset\O_X$.
Let us write $X_i\subset X$ for the corresponding closed subscheme. 
By the Nakayama Lemma, the  complementary open set $U_i\subset X$ is the locus of all points
$a$ having an open neighborhood $U$ such that 
$\shF|U$ can be written as a  quotient of $\O_U^{\oplus i}$.
The crucial point here is that the Fitting ideal endows the closed set 
$X_i\subset X$ with an intrinsic scheme structure, compatible with base change along any $X' \ra X$.
 The inclusions of ideals $\Fitt_i(\shF)\supset\Fitt_{i-1}(\shF)$ correspond to 
inclusions of subschemes $X_i\subset X_{i-1}$.

Now let $k$ be a ground field, and $X$ a $k$-scheme 
of finite type. Then the sheaf of K\"ahler differentials 
$\Omega^1_X = \Omega^1_{X/k}$ 
is coherent, and hence quasicoherent and of finite presentation.
Consider the Fitting ideals $\Fitt_i(\Omega^1_X)$ and the resulting 
closed subschemes $X_i\subset X$.

\begin{proposition}
\label{prop:stable}
Assume that $X$ is a $G$-scheme for some group scheme $G$. 
Then the closed subschemes $X_i\subset X$ are $G$-stable.
\end{proposition}

\proof
We have to check that for any scheme $T$, each $T$-valued point 
$g\in G(T)$ 
stabilizes the base change $X_i \times T$.
Since K\"ahler differentials and Fitting ideals commute with 
base change, it suffices to show that
$X_i$ is stable with respect to any $T$-automorphism 
$g:X \times T \ra X \times T$.
Such an automorphism induces a bijection 
$\varphi:g^*(\Omega^1_{X \times T/T})\ra \Omega^1_{X \times T/T}$,
and thus 
$\Fitt_i(g^*\Omega^1_{X \times T/T}) =  \Fitt_i(\Omega^1_{X \times T/T})$.
On the other hand, we have 
$g^{-1}\Fitt_i(\Omega^1_{X \times T/T}) \cdot \O_{X \times T} 
= \Fitt_i(g^*\Omega^1_{X \times T/T})$,
again because Fitting ideals commute with base change. 
So we see that $g :X \times T \ra X \times T$ 
sends the ideal corresponding to $X_i \times T\subset X \times T$ 
to itself, and thus stabilizes $X_i \times T$.
\qed

\medskip

Let $X$ and $Y$ be schemes that are separated and of finite type, 
and $f:X\ra Y$ be a proper morphism with $f_*(\O_X)=\O_Y$,
and suppose that a connected algebraic group $G$  acts on $X$. 
According to Blanchard's Lemma, there is a unique $G$-action on $Y$ 
making the morphism $f$ equivariant.
Let  $X_i\subset X$ be the closed subscheme defined by 
$\Fitt_i(\Omega^1_X)$,
and denote its schematic image by $Z_i \subset Y$. 

\begin{corollary}
\label{cor:stable}
In the above setting, the closed subschemes 
$Z_i \subset Y$ are $G$-stable.
\end{corollary}

\proof
This follows by combining Lemma \ref{lem:image} and 
Proposition \ref{prop:stable}.
\qed

\section{Computations with Rees rings}
\label{sec:rees}

Let $k$ be a field, and $R=k[x_1,\ldots,x_n]$ be the polynomial
ring in $n\geq 1$ indeterminates $x_i$. 
Fix some integers $1\leq s\leq n$ and  $v_s,\ldots,v_n \geq 1$,
and consider the ideal
$$
\ideala=(x_s^{v_s},\ldots,x_n^{v_n}).
$$
We then form the \emph{Rees ring} $R[\ideala T]$, 
which is the $R$-subalgebra of the polynomial ring $R[T]$
generated by the homogeneous elements $x_s^{v_s}T,\ldots,x_n^{v_n}T$.
We now seek to understand, from a purely algebraic view,   a certain Fitting ideal  
for the K\"ahler differentials 
$\Omega^1_{R[\ideala T]} = \Omega^1_{R[\ideala T]/k}$.  
This bears geometric consequences,
because the homogeneous spectrum of the Rees ring is the blow-up of $\AA^n=\Spec(R)$
with respect to the center $Z=\Spec(R/\ideala)$.

The Rees ring $R[\ideala T]$ is generated as a graded $k$-algebra 
by the homogeneous elements
$$
x_1 ,\ldots,x_n \quadand x_s^{v_s}T,\ldots,x_n^{v_n}T.
$$
Between these generators we  have the obvious relations
$x_i^{v_i}\cdot x_j^{v_j}T = x_j^{v_j}\cdot x_i^{v_i}T$ 
for all $s\leq i< j\leq n$, and these actually generate the ideal of all relations.
Indeed, the canonical morphism $\Sym^\bullet(\ideala)\ra R[\ideala T]$ is bijective,
and the obvious relations generate the ideal of all relations for the symmetric algebra
(see \cite{Micali 1964}, Theorem 1 and Lemma 2).
In turn,    $\Omega^1_{R[\ideala T]}$ is generated as module over the Rees ring by   $2n-s+1$ differentials 
$dx_1, \ldots, dx_n, d(x_s^{v_s}T),\ldots,d(x_n^{v_n}T)$.
These are only subject  to the $\binom{n-s+1}{2}$ relations $d(x_i^{v_i}\cdot x_j^{v_j}T) = d(x_j^{v_j}\cdot x_i^{v_i}T)$.

We now assume that the ground field $k$ has characteristic $p>0$, and  that there is an integer $s \leq l<n$
such that   $p$ divides the the  exponents $v_s,\ldots,v_l$
whereas  $v_{l+1}=\ldots=v_n=1$. This ensures
$$
d(x_i^{v_i})= v_ix_i^{v_i-1}dx_i= \begin{cases}
0 & \text{if $i\leq l$;}\\
dx_i	& \text{else.}
\end{cases}
$$
The  following result computes one Fitting ideal for the K\"ahler differentials outside the locus
defined by the irrelevant ideal $R[\ideala T]_+=(x_s^{v_s}T,\ldots,x_n^{v_n}T)$:

\begin{theorem}
\label{thm:fitting}
Assumptions as above. Then  the two ideals 
$$
\Fitt_{n +l-s+1}(\Omega^1_{R[\ideala T]}) \quadand (x_s^{v_s},\ldots,x_n^{v_n},x_s^{v_s}T,\ldots,x_l^{v_l}T)
$$
in the Rees ring $R[\ideala T]$ induce  the same ideals  in the localization $R[\ideala T]_g$,
for any $g=x_r^{v_r}T$,   $s\leq r\leq n$.
\end{theorem}

\proof
We first reduce the problem to the case $s=1$.  Consider the subrings  $R_0=k[x_1,\ldots,x_{s-1}]$ and $R'=k[x_s,\ldots,x_n]$
inside the Rees ring,  and the ideal  $\ideala'=(x_s^{v_s},\ldots,x_n^{v_n})$ inside $R'$. We then have a tensor product decomposition
$$
R[\ideala T]= k[x_1,\ldots, x_{s-1},x_s,\ldots,x_n, x_s^{\nu_s}T,\ldots,x_n^{\nu_n}T] = R_0\otimes_k R'[\ideala' T]
$$
of the Rees ring, and thus a direct sum decomposition
$$
\Omega^1_{R[\ideala T]} = \left(\Omega^1_{R_0}\otimes_{R_0} R[\ideala T]\right)  \oplus
\left( \Omega^1_{R'[\ideala' T]/k}\otimes_{R'[\ideala' T]} R[\ideala T]\right)
$$
for the K\"ahler differentials (see \cite{Eisenbud 1995}, Proposition 16.5). 
Using that the  summand on the left is free of rank $s-1$, we infer
\[ \Fitt_{n+l-s+1}(\Omega^1_{R[\ideala T]}) = 
\Fitt_{(n-s+1)+(l-s+1)}(\Omega^1_{R'[\ideala' T]})\otimes_{R'[\ideala' T]} R[\ideala T]. \]
This reduces our problem to the case $s=1$.

In this situation, we have   $2n$ generators, which  come in three kinds, namely 
$$
dx_1 ,\ldots,dx_l\quadand  dx_{l+1} ,\ldots,dx_n\quadand d(x_1^{v_1}T),\ldots,d(x_n^{v_n}T).
$$
Likewise, we group the relations into three types, using the product rule and our assumptions on the exponents:  First we have  the $\binom{l}{2}$
relations  
\begin{equation}
\label{minors I}
x_i^{v_i}\cdot d(x_j^{v_j}T) - x_j^{v_j}\cdot d(x_i^{v_i}T),\quad 1\leq i<j\leq l.
\end{equation}
Then there are the  $l(n-l)$ relations
\begin{equation}
\label{minors II}
x_i^{v_i}\cdot d(x_jT) - x_j\cdot d(x_i^{v_i}T) - x_i^{v_i}T\cdot d(x_j),\quad \text{$1\leq i\leq l$ and $ l+1\leq j\leq n$}.  
\end{equation}
Finally,  the remaining $\binom{n-l}{2}$ relations take the form
\begin{equation}
\label{minors III}
x_i\cdot d(x_jT)+ x_jT\cdot dx_i - x_j\cdot d(x_iT) - x_iT\cdot dx_j,\quad l+1\leq i<j\leq n.
\end{equation}
Recall that we assume $p\mid v_i$ for $s\leq i\leq l$, and $v_{l+1}=\ldots=v_n=1$, which indeed enters in the above formulas.
Using these generators and relations, we obtain a presentation of $\Omega^1_{R[\ideala T]}$ by a matrix $P$
of size $2n\times\binom{n}{2}$. Since we have three kinds of generators and three types of relations, one can view
$P$ as a block matrix with $3\cdot 3=9$ blocks,  roughly indicated as follows:
\newcommand{\+}{\phantom{+}}
\begin{equation}
\label{presentation matrix}
P=
\left(
\begin{array}{c|c|c}
&&\\
\cline{1-3}
& 
\begin{matrix}
		\\
		\\
		\\
& -x_i^{v_i}T& \\	
		\\
\end{matrix}  
&
\begin{matrix}
  		\\
& \+x_jT&\\
		\\
& -x_iT& 	\\	
		\\
\end{matrix} 
			\\ 
			\cline{1-3} 
\begin{matrix}
\\
& -x_j^{v_j}\\
& \+x_i^{v_i}\\
		\\
		\\
\end{matrix} 		
& 
\begin{matrix}
\\
& -x_j 	\\
		\\
& \+x_i^{v_i}\\
		\\
\end{matrix} 
& 
\begin{matrix}
\\
		\\
		\\
& -x_j \\
& \+x_i \\
\end{matrix} 
			\\
\end{array}
\right)
\end{equation}
Note that the columns of $P$ are indexed by the 2-element subsets $\{i,j\}\subset\{1,\ldots,n\}$, 
which we regard as $1\leq i<j\leq n$, equipped with the lexicographic order.
Also note that each column  contains either two, three or four non-zero entries, 
and that all three upper blocks are zero. 
 
We have to gain  control  on  the minors of size $2n-(n+l)=n-l$, after localization of $g=x_r^{v_r}T$
with   fixed  $1\leq r\leq n$. The case $l=n-1$ can be dealt with immediately:
Then the minors in question are just the matrix entries,
and the assertion becomes obvious.  Note that in this case, the three blocks to the right are empty.
From now on, we assume that $l\leq n-2$, such that the blocks on the right are non-empty.
 
Suppose    $r\leq l$. After renumbering, we may assume $r=1$. Now look  at the central block:
The $n-l$ rows  indexed by $\{i,j\}$   with  $i=1$ and $l+1\leq  j\leq  n$ yield a scalar submatrix of the form
$x_1^{v_1}T\cdot E_{n-l}$, whose  scalar $ x_1^{v_1}T=g$ is  invertible  in the localization $R[\ideala T]_g$. Here $E_{n-l}$ denotes
the identity matrix of size $(n-l)\times(n-l)$.
Consequently,   the     Fitting ideal becomes the unit ideal upon localization, and the same obviously holds for the ideal
$\idealb=(x_1^{v_1},\ldots,x_n^{v_n},x_1^{v_1}T,\ldots,x_l^{v_l}T)$.

We now suppose $r\geq l+1$. Our  task is to verify the equality
$$
\Fitt_{n+l}(\Omega^1_{R[\ideala T]})_g = \idealb_g 
$$
of ideals in $R[\ideala T]_g$. After renumbering, 
we may assume $r=l+1$,  so that $g=x_{l+1} T$.  
We first check that every generator  of $\idealb$ is contained in the localized Fitting ideal.
The trick   is to  examine the  right middle block:
Removing its first row and keeping the  $n-l-1$ columns    indexed by $\{i,j\}$   with $i=l+1$ and $l+2\leq j\leq n$
yields the  scalar submatrix $-g E_{n-l-1}$.
Now fix some $i\leq l$ and  look at the central block: Its column corresponding to $\{i,l+1\}$  
has as single entry $-x_i^{v_i}T$. We infer that $P$ contains   an upper triangular submatrix   of the form
$$
\begin{pmatrix}
-x_i^{v_i}T 	& \ast\\
0		& -g E_{n-l-1}
\end{pmatrix} 
\in\Mat_{n-d}(R[\ideala T])
$$
and conclude that $x_i^{v_i}T$ is contained in the localized Fitting ideal, for all $1\leq i\leq l$. 
But then also  $x_i^{v_i}=x_{l+1}/g\cdot x_i^{v_i}T$ belongs to it. For $j\geq l+1$,
we use the lower middle block to produce a  submatrix of the form
$$
\begin{pmatrix}
\ast 		& -g E_{n-l-1}\\
-x_j 	& 0
\end{pmatrix} 
\in\Mat_{n-d}(R[\ideala T]).
$$
Now Laplace expansion reveals that $x_j$ belongs to the localized Fitting ideal.

It remains to  verify that the localized Fitting ideal is contained in $\idealb_g$, which is the most difficult part of the argument. 
Note that all but the entries from the right middle block $B$ already belong to $\idealb$. We seek to remove the critical entries
from the matrix by using elementary row and column operations.
For this we have to inspect $B$ in more detail. In light of the lexicographic order on the indices $\{i,j\}$,
it takes  the form
$$
B=
\left(
\begin{array}{ccc|ccc}
x_{l+2}T 	& \cdots	& x_nT 	& 0	& \cdots	& 0\\
\cline{1-6}
-g		&		&		&	&		&\\
		& \ddots	&		&	&	\ast	&\\
		&		& -g		&	&		&		
\end{array}
\right),
$$
with $n-l$ rows and $\binom{n-l}{2}$ columns. Note that the columns to the left have indices $\{i,j\}$ with $i=l+1$
and $l+2\leq j\leq n$.
To proceed, we perform elementary row operations over the localization $R[\ideala T]_g$ 
to remove the entries $x_{l+2}T,\ldots,x_nT$ from the top row.
The crucial observation is that this process does not afflict  the remaining zeros in the top row:
For any index $\{j,j'\}$ with  $l+2\leq j<j'\leq n$ the corresponding column in $B$ contains two non-zero entries, namely $x_{j'}T$ and $-x_jT$.
The combined row operations add  
$$
\frac{x_jT}{-g} \cdot (x_{j'}T) + \frac{x_{j'}T}{-g} \cdot (-x_jT)=0
$$
into  the $\{j,j'\}$-position of the top row. Summing up,  the top row becomes trivial.
Performing additionally  elementary column operations, we make the lower right block zero as well.

Now examine the effect of the above elementary row and column operations on  the whole matrix $P$.
The result is a matrix $P'$ of the form
$$
P'=
\left(
\begin{array}{c|c|c}
0	& 0	& 0\\
\cline{1-3}
\ast	& \ast	& B'\\
\cline{1-3}
\ast	& \ast	& \ast
\end{array}
\right),
$$
where $B'=(\begin{smallmatrix}0&0\\-gE_{n-l-1}& 0 \end{smallmatrix})$, and 
all entries from the $\ast$-blocks   belong to the ideal $\idealb_g$.  
Performing further elementary row and column operations, we may assume that in $P'$ also the neighbors 
to the left and below of the  subblock $-gE_{n-l-1}$ become zero. Clearly, the entries in the $\ast$-blocks
still belong to the ideal $\idealb_g$.
As explained in \cite{Eisenbud 1995}, page 493 our localized ideal
$\Fitt_{n+l}(\Omega^1_{R[\ideala T]})_g$, which
is generated by the $(n-l)$-minors of the matrix $P'$,
is also generated by the entries of the submatrix $Q'$ obtained by removing the rows and columns   passing through the subblock $-gE_{n-l-1}$.
It follows that the localized Fitting ideal must belong to $\idealb_g$.
\qed

\medskip
We now translate the above result into more geometric form: Set $Y=\AA^n=\Spec(R)$ and $Z=\Spec(\bar{R})$, where $\bar{R}=R/\ideala$.
Form  the blow-up $X=\Bl_Z(Y)$  and let $f:X\ra Y$ be the ensuing morphism.
The exceptional divisor $E=f^{-1}(Z)$ is the homogeneous spectrum of 
$$
R[\ideala T]\otimes_R\bar{R}=\bar{R}[x_s^{v_s}T,\ldots,x_n^{v_n}T].
$$
This gives an identification $E=\PP^{n-s}_{\bar{R}}$, and we may regard the generators $x_i^{v_i}T$ as global sections
for the invertible sheaf $\O_X(1)=\O_X(-E)$.

\begin{corollary}
\label{blowing up affine space}
In the above setting, the closed subscheme $X'\subset X$ corresponding to the Fitting ideal 
$\Fitt_{n+l-s}(\Omega^1_X)$
is the linear subscheme  inside the exceptional divisor $E=\PP^{n-s}_{\bar{R}}$  
given by the equations $x_i^{v_i}T=0$ for $s\leq i\leq l$. If furthermore $n\geq l+1$, 
then the schematic image of $X'$
coincides with the center $Z\subset Y$.
\end{corollary}

\proof
Write $S=R[\ideala T]$ for the Rees ring, and fix  one of the homogeneous generators $g=x_i^{v_i}T$ of degree one.
Consider  the localization $S_g$ and its degree-zero part $S_{(g)}$.
The latter defines the basic open set   $D_+(g) = \Spec(S_{(g)})$ 
inside the blow-up $X=\Proj(S)$.
According to \cite{EGA II}, paragraph (2.2.1) the canonical map
$S_{(g)}\otimes_k k[T^{\pm 1}]= S_{(g)}[T^{\pm}]\ra S_g$ given by the assignment $T\mapsto g/1$
is bijective. In turn, we get a  decomposition
$$
\Omega^1_{S_g} = (\Omega^1_{S_{(g)}} \otimes_{S_{(g)}} S_g) \oplus (\Omega^1_{k[T^{\pm 1}]} \otimes_{k[T^{\pm1}]} S_g),
$$
by \cite{Eisenbud 1995}, Proposition 16.5.
The summand on the right is free of rank one, and we conclude that the $j$-th Fitting ideal 
for $\Omega^1_{S_g}$ inside $S_g$ is induced from the $(j-1)$-th Fitting ideal
for $\Omega^1_{S_{(g)}}$ inside the homogeneous localization $S_{(g)}$.

We now apply Theorem \ref{thm:fitting}
with $j=n+l-s+1$, and see that $X'$ is contained in the exceptional divisor $E=\PP^{n-s}_{\bar{R}}$.
In fact, it is a linear subscheme  of codimension $l-s+1$. We thus have $X'=\PP^d_{\bar{R}}$ with
$d=(n-s)-(l-s+1) = n-l-1$. The schematic image is given by the spectrum of $H^0(X',\O_{X'})$.
Under the additional condition $n\geq l+1$ we have $d\geq 0$, and this ring of global sections is indeed $\bar{R}$. 
\qed

\section{The non-smooth case}
\label{sec:ns}

This section contains the proof of Theorem \ref{thm:main} 
in the case that $G$ is non-smooth. 
Then our ground field $k$ has characteristic $p>0$.  
The proof requires some preparation and appears at the end of the section.

Let us first consider the following situation: Suppose that $G$ acts freely 
on some geometrically integral scheme of finite type $V$.  
Then there exists a dense open $G$-stable set $V_0 \subset V$ such that 
the quotient $V_0/G$ exists as a scheme,  $V_0/G$ is of finite type,  
and the projection $V_0 \to V_0/G$ is a $G$-torsor (see 
\cite{SGA 3a}, Expos\'e V,  Th\'eor\`eme 8.1).  
We may also assume that $V_0$ is smooth by using Lemma \ref{lem:smooth}; 
then $V_0/G$ is smooth as well.  

Replacing $V$ with $V_0$,  we may therefore assume that 
$V$ is smooth and the total space of some  $G$-torsor $V \to V/G$, 
where $V/G$ is a smooth scheme of finite type.
We thus find a finite \'etale subscheme $F \subset V/G$.  
The cartesian square 
$$
\begin{CD}
W	@>>> 	F\\
@VVV		@VVV\\
V	@>>>	V/G
\end{CD}
$$
defines a $G$-stable closed subscheme $W \subset V$. 
We now consider the blow-up $U = \Bl_W(V)$, 
and write $f:U\ra V$ for the ensuing morphism.
Define integers by the equations
\begin{equation}
\label{some integers}
n=\dim(V)\quadand s-1=\dim(G)\quadand  l  =\dim_k(\lieg),
\end{equation}
where $\lieg=\Lie(G)$.

\begin{proposition}
\label{image equals center}
In the above setting, let $U'\subset U$ be the closed subscheme defined
by the sheaf of ideals $\Fitt_{n+ l-s}(\Omega^1_U)$, and $V' \subset V$
the schematic image of $U'$. If $n\geq l+1$ then $V'$ coincides with 
the center $W \subset V$.
\end{proposition}

\proof
To verify this equality of closed subschemes, 
it suffices to treat the case that $k$ is algebraically closed.
Then $F$ is a finite sum of copies of $\Spec(k)$.
By passing to suitable open sets, it suffices to treat the case that $F = \Spec(k)$.
Fix a closed point $a\in V$. Our task is to check that  
the two closed subschemes $V',W\subset V$ give the same ideal 
in the formal completion $\O_{V,a}^\wedge$.
To proceed, note that    $s=\dim(G)+1\leq \dim_k(\lieg) =l$ because $G$ is non-smooth. Thus $n +l-s\geq n$.
Since  $\Omega^1_V$ is locally free of rank $n$ we see that 
$\Fitt_{n +l-s}(\Omega^1_V) = \catO_V$.
It follows that $V' \subset W$, at least as closed sets, 
and this already settles the case $a\notin W$.

Suppose now that $a\in W$. Then the orbit map $G \to U$, 
$g \mapsto g \cdot a$ gives an identification $G=W$. 
Set $R=\O_{V,a}^\wedge$ and $A=\O_{W,a}^\wedge=\O_{G,e}^\wedge$. 
According to \cite{Demazure; Gabriel 1970}, Chapter III,  Corollary 6.4,
we have $A=k[[u_1,\ldots,u_l]]/(u_s^{v_s},\ldots, u_l^{v_l})$ 
for some $1\leq s\leq l$ and some $p$-powers $v_s,\ldots,v_l\neq 1$.
Moreover, $R$ is isomorphic to a formal power series ring 
in $n$ indeterminates over the field $k$.
By Lemma \ref{complete intersection} below, there is a regular
system of parameters 
$x_1,\ldots,x_n\in R$
such  that the kernel of the surjection $R\ra A$ is generated by $x_s^{v_s},\ldots, x_l^{v_l},x_{l+1},\ldots,x_n$.
Setting $v_{l+1}=\ldots=v_n=1$ we get 
$$
\O_{V,a}^\wedge=k[[x_1,\ldots,x_n]]\quadand \O_{W,a}^\wedge=k[[x_1,\ldots,x_n]]/(x_s^{v_s},\ldots, x_n^{v_n}).
$$
This is exactly  the situation in Corollary \ref{blowing up affine space}, 
after base change from the polynomial ring $k[x_1,\ldots,x_n]$ to the formal power series ring. 
Since the formations of blow-ups and schematic images commute with flat
base change, and formation of Fitting ideals commutes with arbitrary base change, 
the corollary ensures that the center and schematic image coincide in 
$\Spec(\O_{V,a}^\wedge)$. 
\qed

\medskip
In the preceding proof we have used some facts on formal power series rings:
Suppose $A$ is a local noetherian ring of the form $A=k[[u_1,\ldots,u_l]]/(f_s,\ldots,f_l)$,
where $u_1,\ldots, u_l$ are indeterminates  and  $f_s,\ldots,f_l$ are formal power series 
that form a regular sequence contained in   $\maxid_{R}^2$.
Then   $\dim(A)=l-(l-s+1) = s-1$, whereas the \emph{embedding dimension} equals 
$\edim(A)=\dim_k(\maxid_A/\maxid_A^2)=l$.
Now let $\varphi:R\ra A$ be any  surjection from a  formal power series ring $R$   
in $n\geq l$ indeterminates over the field $k$. Then   
$\ideala=\Kernel(\varphi)$ is generated by a regular sequence contained in $\maxid_R$,
according to \cite{EGA IV.4}, Proposition 19.3.2. 
The following observation gives additional information:

\begin{lemma}
\label{complete intersection}
In the above situation, there exists a regular 
system of parameters $x_1,\ldots,x_n\in R$ 
such that $\ideala$ is generated by
$f_i(x_1,\ldots,x_l)$ with $s\leq i\leq l$, together with the $x_j$ with $l+1\leq j\leq n$.
\end{lemma}

\proof
Consider the short exact sequence of vector spaces
$$
0\lra \ideala/(\ideala\cap \maxid_R^2)\lra \maxid_R/\maxid_R^2\lra \maxid_A/\maxid_A^2\lra 0.
$$
Write $\bar{u}_i\in A$ for the classes of the indeterminates $u_i$.
Choose $x_1,\ldots,x_l\in \maxid_R$ mapping to   $\bar{u}_1,\ldots,\bar{u}_l\in\maxid_A$.
 The latter form a basis of the cotangent space $\maxid_A/\maxid_A^2$,
so the former stay linearly independent in $\maxid_R/\maxid_R^2$.
Choose $x_{l+1},\ldots,x_n\in R$ that belong to $\ideala$ 
and that yield a basis of $\ideala/(\ideala\cap \maxid_R^2)$.
Then $x_1,\ldots,x_n$ form a basis in $\maxid_R/\maxid_R^2$, hence constitute 
a regular system of parameters in $R$.
Set $g_i=f_i(x_1,\ldots,x_l)$ for $s\leq i\leq l$, and  form the ideal $\idealb=(g_s,\ldots,g_l,x_{l+1},\ldots,x_n)$.
Clearly $\idealb\subset \ideala$, so the surjection $\varphi:R\ra A$ induces a map
$$
k[[x_1,\ldots,x_r]]/(g_s,\ldots,g_l)=R/\idealb\lra A= k[[u_1,\ldots,u_l]]/(f_s,\ldots,f_l)
$$
This  sends generators to generators, and relations to relations, whence is bijective.
Hence the kernel $\ideala/\idealb$ vanishes, and thus  $\idealb=\ideala$.
\qed

\medskip

\emph{Proof of Theorem \ref{thm:main} for non-smooth $G$.}
According to Proposition \ref{prop:subgroup}, 
the algebraic group $G$ embeds into a smooth connected algebraic 
group $H$. By Section \ref{sec:smooth}, 
there is a projective scheme $Y$ with $\Aut^0_Y=H$,
and we may furthermore assume that $Y$ is normal and geometrically integral, 
and that there is a dense
open set $V\subset Y$ that is $H$-stable, with free $H$-action.  
We may also assume that $\dim(Y)= \max(2\dim(H),3)$.
Choose $W \subset V$ as described before Proposition \ref{image equals center},   
write $Z\subset Y$ for its schematic closure,
let $X=\Bl_Z(Y)$ be the blow-up, and   
$f:X\ra Y$ be the resulting birational morphism. 
Since $Y$ is normal, 
the map $\O_Y\ra f_*(\O_X)$ is bijective, according to Zariski's Main Theorem.

By Blanchard's Lemma, there is a unique action of $\Aut_X^0$ on $Y$ 
making the morphism $f:X\ra Y$ equivariant. 
As $f$ is birational, 
we thus get an inclusion $\Aut_X^0 \subset \Aut_Y^0 = H$.
On the other hand, since the $G$-action on $V$ stabilizes $W$, 
the $G$-action on $Y$ stabilizes the schematic closure $Z$
(Lemma \ref{lem:image}). So this action lifts uniquely
to an action on $X$ (see \cite{Martin 2022}, Proposition 2.7). 
This yields an inclusion $G \subset \Aut^0_X$, 
and our task is to show that this is an equality.
For this we may assume that $k$ is algebraically closed.

As in \eqref{some integers}, we define integers  
$n=\dim(X)$ and $s-1=\dim(G)$ 
and $l=\dim(\lieg)$, where $\lieg=\Lie(G)$.
Consider the open set $U=f^{-1}(V)$, and let $U'\subset U$ 
be the closed subscheme corresponding 
to the sheaf of Fitting ideals $\Fitt_{n+ l-s}(\Omega^1_U)$. 
We have $n=\dim(X) \geq  2\dim(H)\geq 2l \geq l+1$ by our assumptions.
According to Proposition \ref{image equals center}, 
the schematic image of $U'$ in $V$ coincides with $W$.
Since $\Aut^0_X$ stabilizes the open set $U=f^{-1}(V)$, it must also stabilize $U'$.
In turn, its action on $V$ stabilizes $W$. 
Fix a closed point $a\in W$, and set $G'=\Aut_X^0$.
Recall that $H=\Aut_Y^0$ acts freely on $V$. 
So the orbit maps give inclusions of schemes 
$G\subset G'=G'\cdot a\subset G\cdot a=G$, thus $G=G'$.
\qed

\begin{remark}
\label{rem:nonnormal}
In general, the scheme $X$ constructed in the above proof 
is non-normal. Consider indeed the finite group scheme
$G = \alpha_p \times \alpha_{p^2}$. Then
$G$ is a subgroup scheme of the smooth connected algebraic 
group $H = \GG_a \times \GG_a$. So the above construction 
yields a $G$-scheme $X$ which contains a $G$-stable dense 
open set, isomorphic to the blow-up of 
a dense open set of $H \times H = \AA^4$ along the 
$G$-orbit of a $k$-rational point. To show that $X$ is
non-normal, it suffices to check that the blow-up of
$\AA^4$ along the $G$-orbit of the origin is non-normal
along the exceptional divisor. We then have
$R = k[x_1,x_2,x_3,x_4]$, 
$\ideala = (x_3^p, x_4^{p^2})$ and
$S = R[\ideala T] 
= k[x_1,x_2,x_3,x_4, x_3^p T, x_4^{p^2} T]$.
Let $g = x_3^p T$, then the degree-$0$ localization
$S_{(g)}$ satisfies
\[ S_{(g)} = k[x_1,x_2,x_3,x_4,x_4^{p^2}/x_3^p]. \]
The normalization of $S_{(g)}$ in its fraction field contains 
$x_4^p/x_3$, which is not in $S_{(g)}$. So $S_{(g)}$
is indeed non-normal along the exceptional divisor
(the zero subscheme of $x_3$).
\end{remark}

\section{Conflicts of interest}

The authors have no conflicts of interest to declare that are relevant to this article.


\end{document}